\documentclass[cleveref,a4paper,numberwithinsect]{lipics-v2021}
\usepackage{cite} 
\usepackage{amsmath}
\usepackage{amssymb}  

\usepackage{tikz}
\usepackage{todonotes}
\tikzstyle{every picture} = [>=latex]
\usetikzlibrary{calc,arrows,decorations.markings,decorations.pathreplacing,quotes,fit,shapes,positioning}

\def\ca#1{{\mathcal#1}}

\title{Sparse Graphs of Twin-width 2 Have Bounded Tree-width}

\author{Benjamin Bergougnoux}{University of Warsaw, Poland}{benjamin.bergougnoux@mimuw.edu.pl}{https://orcid.org/0000-0002-6270-3663}{}
\author{Jakub Gajarsk\'y}{University of Warsaw, Poland}{***}{https://orcid.org/***}{}
\author{Grzegorz Gu\'spiel}{Masaryk University, Brno, Czech republic}{guspiel@fi.muni.cz}{https://orcid.org/0000-0002-3303-8107}{}
\author{Petr Hlin\v{e}n\'y}{Masaryk University, Brno, Czech republic}{hlineny@fi.muni.cz}{https://orcid.org/0000-0003-2125-1514}{}
\author{Filip Pokr\'yvka}{Masaryk University, Brno, Czech republic}{xpokryvk@fi.muni.cz}{https://orcid.org/0000-0003-1212-4927}{}
\author{Marek Soko{\l}owski}{University of Warsaw, Poland}{marek.sokolowski@mimuw.edu.pl}{https://orcid.org/0000-0001-8309-0141}{}

\funding{J.\ Gajarsk\'y and M.\ Soko{\l}owski have received funding from the European Research Council (ERC) (grant agreement No 948057 -- {\sc BOBR}).
\begin{picture}(0,0)
\put(220,-10)
{\hbox{\includegraphics[width=40px]{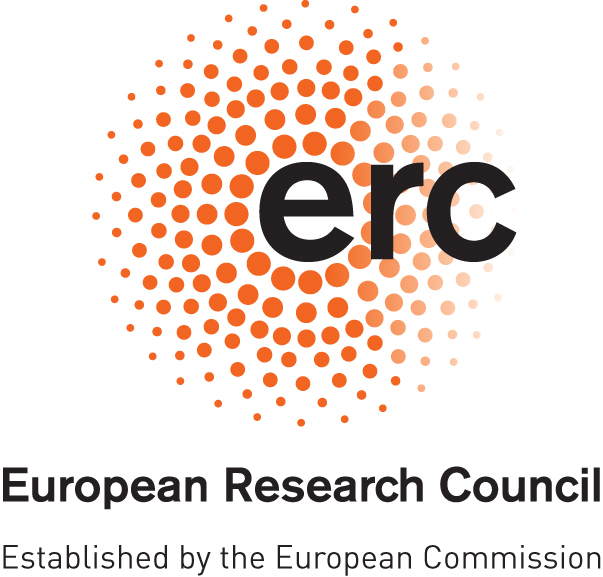}}}
\put(210,-70)
{\hbox{\includegraphics[width=60px]{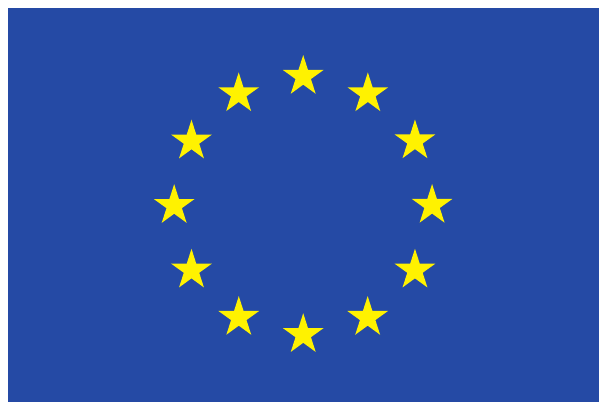}}}
\end{picture}}
  
\authorrunning{B.\ Bergougnoux, J.\ Gajarsk\'y, G.\ Gu\'spiel, P.\ Hlin\v{e}n\'y, F.\ Pokr\'yvka, M.\ Soko{\l}owski}
\Copyright{Benjamin Bergougnoux, Jakub Gajarsk\'y, Grzegorz Gu\'spiel, Petr Hlin\v{e}n\'y, Filip Pokr\'yvka and Marek Soko{\l}owski}



\hideLIPIcs
\nolinenumbers

\begin{document}
\maketitle

\begin{abstract}
Twin-width is a structural width parameter introduced by Bonnet, Kim, Thomass\'e and Watrigant [FOCS 2020].
Very briefly, its essence is a gradual reduction (a contraction sequence) of the given graph down to a single vertex while maintaining
limited difference of neighbourhoods of the vertices, and it can be seen as widely generalizing several other traditional structural parameters.
Having such a sequence at hand allows to solve many otherwise hard problems efficiently.
Our paper focuses on a comparison of twin-width to the more traditional tree-width on sparse graphs.
Namely, we prove that if a graph $G$ of twin-width at most~$2$ contains no $K_{t,t}$ subgraph for some integer $t$,
then the tree-width of $G$ is bounded by a polynomial function of~$t$.
As a consequence, for any sparse graph class $\ca C$ we obtain a polynomial time algorithm which for any input graph $G \in \ca C$ either outputs a contraction sequence of width at most $c$ (where $c$ depends only on $\ca C$), or correctly outputs that $G$ has twin-width more than $2$.
On the other hand, we present an easy example of a graph class of twin-width $3$ with unbounded tree-width, showing that our result cannot be extended to higher values of twin-width.
%
%
\end{abstract}

\section{Introduction}

Twin-width is a relatively new structural width measure of graphs and relational structures introduced in 2020 by Bonnet, Kim, Thomass\'e and Watrigant~\cite{DBLP:conf/focs/Bonnet0TW20}.
Informally, the twin-width of a graph measures how diverse the neighbourhoods of the graph vertices are.
For instance, \emph{cographs}---the graphs which can be built from singleton vertices by repeated operations of a disjoint union and taking the complement---
are exactly the graphs of twin-width 0, which means that the graph can be brought down to a single vertex by successively identifying twin vertices.
(Two vertices $x$ and $y$ are called \emph{twins} in a graph $G$ if they have the same neighbours in $V(G) \setminus \{x,y\}$.)
Hence the name, \emph{twin-width}, for the parameter.

Importance of this new concept is clearly witnessed by numerous recent papers on the topic, such as the follow-up series
\cite{DBLP:conf/soda/BonnetGKTW21,DBLP:conf/icalp/BonnetG0TW21,DBLP:conf/stoc/BonnetGMSTT22,DBLP:conf/soda/BonnetKRT22,DBLP:conf/iwpec/BonnetC0K0T22,DBLP:journals/jacm/BonnetKTW22}
and more related research papers represented by, e.g.,
\cite{DBLP:journals/corr/abs-2110-03957,DBLP:conf/iwpec/BalabanH21,DBLP:journals/corr/abs-2102-06880,DBLP:conf/icalp/BergeBD22,DBLP:conf/wg/BalabanHJ22,DBLP:conf/icalp/GajarskyPPT22,DBLP:conf/stacs/PilipczukSZ22,DBLP:journals/corr/abs-2210-08620}.

In particular, twin-width, as a structural width parameter, has naturally algorithmic applications in the parameterized complexity area.
Among the most important ones we mention that the first order (FO) model checking problem --
that is, deciding whether a fixed first-order sentence holds in an input graph -- can be solved in linear FPT-time \cite{DBLP:journals/jacm/BonnetKTW22}.
This and other algorithmic applications assume that a contraction sequence of bounded width is given alongside with the input graph,
since in general we do not know how to construct such an sequence efficiently. 
Consequently, finding an efficient algorithm for computing twin-width of input graph $G$ together with its twin-width decomposition (known as a \emph{contraction sequence} of $G$) is a central problem in the area.
Currently, very little is known about this problem.
It is known that one can check whether graph has twin-width $0$ and $1$ and compute  the corresponding contraction-sequence in polynomial time~\cite{DBLP:journals/algorithmica/BonnetKRTW22}. On the other hand, in general the problem of deciding the exact value of twin-width is NP-hard, and in particular,  deciding whether a graph has twin-width~$4$ is NP-hard~\cite{DBLP:conf/icalp/BergeBD22}.
This means that even for fixed $k$, the best one can hope for is an approximation algorithm which for given input graph $G$ either correctly outputs that $G$ has twin-width more than $k$ or produces a contraction sequence of $G$ of width at most $k'$, where $k'$ is some fixed number with $k < k'$.
However, to the best of our knowledge, no such efficient algorithm is currently known, even for graph classes of twin-width $2$.


The importance and popularity of twin-width largely stems from the fact that it generalizes many well-known graph-theoretic concepts. For example, graph classes of bounded clique-width, planar graphs and more generally graph classes defined by excluding a fixed minor, posets of bounded width and various subclasses of geometric graphs have been shown to have bounded twin-width~\cite{DBLP:journals/jacm/BonnetKTW22,geom}. In many cases, the proof of having bounded twin-width is effective, meaning that for such class there exists $d$ and a polynomial algorithm which takes a graph $G$ as input and outputs a contraction sequence of $G$ of width at most $d$.

%

%
%

Apart from establishing that some well-known graph classes have bounded twin-width, there are also results relating twin-width to various graph classes from the other direction -- one considers graph classes of bounded bounded twin-width which are restricted in some sense, and proves that such classes fall within some well-studied framework.
Prime examples of this approach are results stating that  
$K_{t,t}$-free graph classes of bounded twin-width have bounded expansion~\cite{DBLP:conf/soda/BonnetGKTW21,sparse_tww} or that stable graph classes of bounded twin-width have structurally bounded expansion~\cite{stable_tww}.  Results of this form allow us to use existing structural and algorithmic tools which exist for (structurally) bounded expansion and apply them to restricted classes of bounded twin-width.

Our research is motivated by the following conjecture, which also fits into this line of research and which we learned about from É. Bonnet.

\begin{conjecture}
\label{conj}
Let $\ca C$ be a class of graphs of twin-width at most $3$ such that there exists $t$ such that no $G \in \C$ contains $K_{t,t}$ as a subgraph. Then $\ca C$ has bounded tree-width.
\end{conjecture}

For graph classes of twin-width at most one can easily argue that Conjecture~\ref{conj} is true as follows.
It is known \cite{DBLP:conf/soda/BonnetKRT22} that if a graph class $\ca C$ is such that every $G \in \ca C$ has a contraction sequence in which red components (see Section~\ref{sec:prelims}) have bounded size, then $\ca C$  has bounded clique-width. Since in a contraction sequence of width one each red component has size one, this implies that graph classes of twin-width one have bounded clique-width.  In combination with the fact that graph classes of bounded clique-width which exclude $K_{t,t}$ as a subgraph have bounded tree-width, this proves the result.

However, for twin-width 2 or 3 we cannot assume that red components of trigraphs occurring in contraction sequences have bounded size, and so to attack Conjecture~\ref{conj} one has to employ a more fine-tuned approach by directly analyzing contraction sequences.

\paragraph*{Our contribution}

We confirm Conjecture~\ref{conj} for the case of graph classes of twin-width $2$ and disprove the conjecture for graph classes of twin-width $3$. Namely on the positive side we prove the following.

\begin{theorem}\label{thm:main}
Let $t$ be an integer and let $\ca C$ be a class of graphs of twin-width at most~$2$ such that no $G \in \ca C$ contains $K_{t,t}$ as a subgraph.
Then the tree-width of $\ca C$ is bounded by a polynomial function of~$t$, namely at most~$\ca O(t^{20})$.
\end{theorem}





Theorem~\ref{thm:main} allows us to apply the existing tools for bounded tree-width to sparse classes of twin-width $2$, making many hard problems efficiently solvable on such classes. 
Moreover, Theorem~\ref{thm:main} also leads to the following algorithmic corollary.

\begin{corollary}\label{cor:tww2seq}
Let $\ca C$ be a class of graphs such that there exists $t$ such that no $G \in \ca C$ contains $K_{t,t}$ as a subgraph. 
Then there exists a constant $c$ depending on~$\ca C$, and a polynomial-time algorithm which for every $G \in \ca C$ either outputs a contraction sequence of $G$ of width at most~$c$, or correctly outputs that $G$ has twin-width more than $2$.
\end{corollary}

On the negative side, we exhibit an example of a graph class  $\ca C$ of twin-width $3$ such that no $G \in \ca C$ contains a $K_{2,2}\simeq C_4$ subgraph, and that $\ca C$ has unbounded tree-width.

\section{Related Definitions and Tools}
\label{sec:prelims}
We use standard graph-theoretic terminology and notation. All graphs considered in this paper are finite and simple, i.e., without loops and multiple edges.
For the sake of completeness, we include the following folklore definition and tool.

\begin{definition}[Tree-width~\cite{rs86}]
\label{def:tw}
A {\em tree-decomposition} of a graph $G$ is a pair $(X,T)$ where $T$ is a tree, whose vertices we call nodes, and $X = \{X_i~|~i \in V(T)\}$ is a collection of subsets of $V(G)$ such that 
\begin{enumerate}
\item $\bigcup_{i \in V(T)} X_i = V(G)$,
\item for each edge $vw \in E(G)$, there is $i \in V(T)$ such that $v, w \in X_i$ and,
\item for each $v \in V(G)$ the set of nodes $\{i~|~v \in X_i\}$ forms a subtree of $T$.
\end{enumerate}
The width of tree-decomposition $(\{X_i~|~i \in V(T)\}, T)$ is equal to 
$max_{i \in V(T)} \{\left|X_i\right| - 1\}$. The \emph{tree-width} of a graph $G$ is the minimum width over all tree-decompositions of $G$. 
\end{definition}

For a graph $G$ and $A, B \subseteq V(G)$, an $A-B$ path is a path with one endvertex in $A$ and the other endvertex in $B$. We will use Menger's theorem (see for example~\cite{diestel}).
\begin{theorem}[Menger's theorem]
Let $G$ be a graph and $A,B \subseteq V(G)$. The minimum size of a set $S \subseteq V(G)$ such that there is no $A-B$ path in $G$ is equal to the maximum number of disjoint $A-B$ paths in $G$.
\end{theorem}

For an integer $N$, a graph $H$ is called an $N\times N$ \emph{wall} (or \emph{hexagonal grid})
if it consists of $N$ disjoint paths $P_1,\ldots,P_N$, each $P_i$ with vertices $v^i_1,\dots,v^i_N$ in this order in $P_i$, together with the edges given by the following rule: if both $i,j \in \{1,\ldots, N\}$ have the same parity and $i < N$, then $v^i_j$ is adjacent to $v^{i+1}_j$.

For an integer $N$, a graph $H$ is called an $N\times N$ \emph{cubic mesh} if the maximum degree of $H$ is $3$ and the following holds;
we can write $H=Q_1\cup Q_2$, where each $Q_i$, $i=1,2$, is formed as a vertex-disjoint union of $N$ paths, the ends of paths of $Q_i$ are disjoint from $V(Q_{3-i})$ for $i=1,2$,
and every component--path of $Q_1$ intersects every component--path of~$Q_2$ in precisely one (common) subpath.
These paths of $Q_1$ (resp.~of $Q_2$) are called the \emph{rows} (resp.~\emph{columns}) of the mesh~$H$,
and the vertices of $H$ of degree $3$ are called the \emph{branching vertices} of~$H$.
The intersection of any row and any column of $H$ must be a subpath of nonzero length since~$\Delta(H)=3$, and it contributes two branching vertices.

\begin{observation}\label{obs:wall2mesh}
Any subdivision of the classical $(2N+2)\times(2N+2)$ wall (hexagonal grid) contains an $N\times N$ cubic mesh as a subgraph.
See~\Cref{fig:wall}.
\end{observation}

\begin{figure}[tbh]
	\centering
	\includegraphics[width=0.6\linewidth]{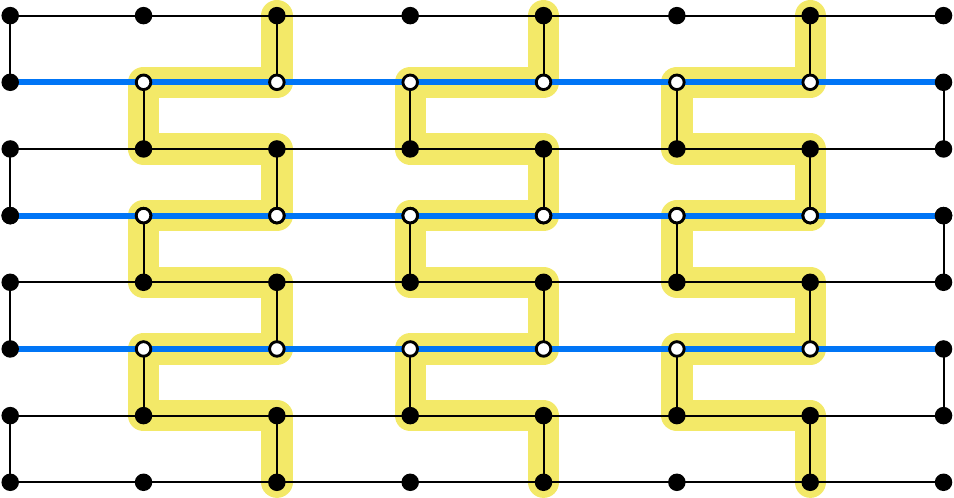}
	\caption{Example of an $8\times 8$ wall. We obtain a $3\times 3$ cubic mesh by considering the yellow paths as columns and the blue horizontal paths as rows. White filled vertices are the branching vertices of this cubic mesh.}
	\label{fig:wall}
\end{figure}

The notion of \emph{twin-width} can in general be considered over arbitrary binary relational structures of a finite signature,
but here we will define it and deal with it for only finite \emph{simple graphs}.
It is based on the following concept.

A \emph{trigraph} is a simple graph $G$ in which some edges are marked as \emph{red}, and with respect to the red edges only, 
we naturally speak about \emph{red neighbours} and \emph{red degree} in~$G$ 
(while the terms neighbour and degree without `red' refer to all edges of $G$ inclusive of red ones, and edges which are not red are also called \emph{black}). 
For a pair of (possibly not adjacent) vertices $x_1,x_2\in V(G)$, we define a \emph{contraction} of the pair $x_1,x_2$ as the operation
creating a trigraph $G'$ which is the same as $G$ except that $x_1,x_2$ are replaced with a new vertex $x_0$ (said to \emph{stem from $x_1,x_2$}) such that:
\begin{itemize}
\item 
the (full) neighbourhood of $x_0$ in $G'$ (i.e., including the red neighbours), denoted by $N_{G'}(x_0)$,
equals the union of the neighbourhoods $N_G(x_1)$ of $x_1$ and $N_G(x_2)$ of $x_2$ in $G$ except $x_1,x_2$ themselves, that is,
$N_{G'}(x_0)=(N_G(x_1)\cup N_G(x_2))\setminus\{x_1,x_2\}$, and
\item 
the red neighbours of $x_0$, denoted here by $N_{G'}^r(x_0)$, inherit all red neighbours of $x_1$ and of $x_2$ and add those in $N_G(x_1)\Delta N_G(x_2)$,
that is, $N_{G'}^r(x_0)=\big(N_{G}^r(x_1)\cup N_G^r(x_2)\cup(N_G(x_1)\Delta N_G(x_2))\big)\setminus\{x_1,x_2\}$, where $\Delta$ denotes the~symmetric set difference.
\end{itemize}

\begin{definition}[Contraction sequence and twin-width~\cite{DBLP:conf/focs/Bonnet0TW20}]\label{def:twinwidth}
A \emph{contraction sequence} of a trigraph $G$ is a sequence of successive contractions turning~$G$ into a single vertex,
and its \emph{width} $d$ is the maximum red degree of any vertex in any trigraph of the sequence.
We say that $G$ has or \emph{admits a $d$-contraction sequence} if there is a contraction sequence of $G$ of width at most~$d$.
The \emph{twin-width} of a trigraph $G$ is the minimum width over all possible contraction sequences of~$G$.

To define a contraction sequence and the twin-width of an ordinary graph~$G$, we consider $G$ as a trigraph with no red edges.
In a summary, a graph has twin-width at most $d$, if and only if it admits a $d$-contraction sequence.
\end{definition}

For our purpose, the following ``inverted'' view of a contraction sequence will be useful.
\begin{definition}[Uncontraction sequence]\label{def:uncontraction}
A \emph{partitioned graph} is a graph $G$ associated with an unordered vertex partition $\ca P=(P_1,\ldots,P_m)$ of~$V(G)$.
The \emph{partitioned trigraph} of $(G,\ca P)$ is a trigraph $H$ on the vertex set $V(H)=\ca P$ such that $\{P_1,P_2\}\in E(H)$
if and only if $G$ contains an edge from $P_1$ to $P_2$, and such edge $\{P_1,P_2\}$ is red if and only if \emph{not} all pairs of $P_1\times P_2$ are edges of~$G$.
An \emph{uncontraction sequence} of an $n$-vertex graph $G$ is a sequence of partitioned graphs $(G,\ca P^i)$ for $i=1,\ldots,n$,
where $\ca P^1=\{V(G)\}$ and $\ca P^n$ is the partition of $V(G)$ into singletons, and for $1<i\leq n$ the partition $\ca P^i$ is obtained
from $\ca P^{i-1}$ by splitting an arbitrary one part of $\ca P^{i-1}$ into two.
\end{definition}

It is easy to observe that if $G_0=G, G_1,\ldots, G_{n-1}$ is a contraction sequence of the $n$-vertex graph $G$,
then for the trigraph $G_{n-k}$ (on $k$ vertices) we may form the corresponding vertex partition $\ca P^{k}=(P_1^{k},\ldots,P_{k}^{k})$ of~$V(G)$
by setting $P_i^k$ to be the set of vertices of $G$ contracted into $w_i\in V(G_{n-k})$, $i=1,\ldots,k$.
The partitioned trigraph of $(G,\ca P^k)$ is hence isomorphic to the trigraph $G_{n-k}$ of the former contraction sequence,
and these two possible approaches to twin-width in \Cref{def:twinwidth} and \Cref{def:uncontraction} exactly coincide.


\section{Proof of Theorem~\ref{thm:main}}

Our proof proceeds by contradiction. Thus, we assume that there is a $K_{t,t}$-free graph $G$ of large tree-width which has twin-width at most $2$.
We then proceed in two steps:
\begin{itemize}\parskip3pt
\item \textbf{Step I:} Since $G$ has large tree-width, it has to contain a subdivision of a large wall, and hence a large cubic mesh, as a subgraph. 
Using this and the assumption that $G$ has an uncontraction sequence of width at most $2$, we show that there has to be a point in time during the uncontraction sequence such that there are four parts $X_1,X_2,X_3, X_4$ which form a red path in the corresponding trigraph and there are many disjoint paths of $G$ fully contained in $X_1\cup X_2 \cup X_3 \cup X_4$ with one endpoint in $X_1$ and end the other in $X_4$.
\item \textbf{Step II:} Using a carefully chosen invariant we show that the structure found in Step I can be maintained indefinitely during the subsequent steps of the uncontraction sequence (still of width at most~$2$). 
This yields a contradiction, since the final partition of the uncontraction sequence consists of singletons and there are no red edges.
\end{itemize}

So, we actually prove the following alternative formulation of \Cref{thm:main}.
\begin{theorem}\label{thm:mainalt}
If $G$ is a simple graph not containing a $K_{t,t}$ subgraph, but containing a subgraph (not necessarily induced) $H\subseteq G$ which is an $N\times N$ cubic mesh for $N=16\cdot(13t)^2$,
then $G$ is of twin-width at least~$3$.
\end{theorem}

\subsection{Proving Step I}

We start with a trivial claim which is crucial in our arguments:
\begin{observation}\label{obs:rededgeonly}
Assume a partitioned graph $(G,\ca P)$ such that $G$ has no $K_{t,t}$ subgraph, and $X_1,X_2\in\ca P$ such that $|X_1|,|X_2|\geq t$.
Then the partitioned trigraph of $(G,\ca P)$ cannot contain a black edge $\{X_1,X_2\}$.
In other words, whenever $G$ has an edge from $X_1$ to $X_2$, there is a red edge $\{X_1,X_2\}$ in the partitioned trigraph of $(G,\ca P)$.
\end{observation}

The first step in the proof of \Cref{thm:main} is precisely formulated and proved next.


\begin{lemma}\label{lem:step1}
Let $t$ be an integer.
Assume that $G$ is a simple graph not containing a $K_{t,t}$ subgraph,
but containing a subgraph (not necessarily induced) $H\subseteq G$ which is an $N\times N$ cubic mesh, where $N=16\cdot(13t)^2$.
If there is an uncontraction sequence $(G,\ca P^i)$ for $i=1,\ldots,|V(G)|$ of width at most~$2$,
then there exists $m\in\{1,\ldots,|V(G)|\}$ such that the following holds:
There are parts $X_1,X_2,X_3,X_4\in\ca P^m$ which induce a red path in this order in the partitioned trigraph of $(G,\ca P^m)$,
there is no edge between $X_1$ and $X_4$,
and the set $X:=X_1\cup X_2\cup X_3\cup X_4$ induces in $G$ a subgraph containing $s\geq 4t$ pairwise vertex-disjoint paths from $X_1$ to~$X_4$.
\end{lemma}


\begin{proof}
Set $k := 13t$.
Let $m\in\{1,\ldots,|V(G)|\}$ be the least index such that every part of $\ca P^m$ contains less than $4k^2$ branching vertices of~$H$.
Then there is a part $Z\in\ca P^m$ such that $Z$ contains at least $2k^2$ branching vertices of~$H$ (one of the two having resulted by the last splitting before~$\ca P^m$).
If $Z$ intersected less than $k$ rows and less than $k$ columns of $H$, by the condition on intersecting rows and columns we would have less than $2k^2$ branching vertices in $Z$ altogether.
Hence, up to symmetry, we may assume that $Z$ hits at least $k$ rows of $H$. Moreover, since $N= 16\cdot(13t)^2 = 16\cdot k^2$, we have that $Z$ contains branching vertices of less than $4k^2=N/4$ rows of $H$.

Let $L_2, L_1 ,Z,R_1,R_2\in\ca P^m$ denote the (at most) four parts that are connected to $Z$ by a red path of length $\leq2$ in the partitioned trigraph of $(G,\ca P^m)$.
That is, we have a red path on, $(L_2,L_1,Z,R_1,R_2)$ in this order (for simplicity, we silently ignore if some of the parts do not exist).
Again, each of $L_1,L_2,R_1,R_2$ contains branching vertices of less than $4k^2=N/4$ rows of~$H$.
Altogether, every row of $H$ hit by $Z$ has at least $2N-5\cdot N/4=3\cdot N/4>0$ branching vertices not contained in $\bigcup\ca Z$ where $\ca Z=\{L_2,L_1,Z,R_1,R_2\}$.
We have thus got, as subpaths of the rows hit by $Z$, a collection $\ca Q = \{Q_1,\ldots,Q_k\}$ of $k$ vertex-disjoint paths in $G$ which connect $Z$ to parts in $\ca P^m\setminus\ca Z$.
Let the paths in $\ca Q$ be chosen as inclusion-minimal and for any $Q_i\in \ca Q$ let  $a_i$ and $b_i$ be the endpoints of $Q_i$.

We now proceed assuming that all $L_1,L_2,R_1,R_2$ have size at least $t$, the case when one of $\{L_1,L_2\}$ or one of $\{R_1,R_2\}$ is smaller than $t$ is addressed below.
Every path $Q\in\ca Q$, by its minimality, connects a vertex of $Z$ to a vertex of $Y\in\ca P^m\setminus\ca Z$ where $Y$ is a neighbour (red or black) of the set $\ca Z$ in the partitioned trigraph of $(G,\ca P^m)$.
For any $X\in\{L_2,L_1,Z,R_1,R_2\}$, since $|X|>t$, the union of the parts adjacent to $X$ by black edges in the partitioned trigraph is of cardinality less than~$t$, or we have a $K_{t,t}$ subgraph.
Together, at most $5t$ of the paths of $\ca Q$ in $G$ may end in parts $Y\in\ca P^m\setminus\ca Z$ which are not red neighbours of $\ca Z$ in the partitioned trigraph.
Since we have at most two red neighbours of parts of $\ca Z$ among the parts of $\ca P^m\setminus\ca Z$ (namely, the ``outside'' red neighbours of $L_2$ and~of~$R_2$),
up to symmetry, the (unique) red neighbour $L_3\in\ca P^m\setminus\ca Z$ of $L_2$ contains ends of at least $(k-5t)/2=4t$ of the paths in~$\ca Q$.
In particular, $L_3$ has to exist, and $L_3$ is not adjacent to $Z$ in the partitioned trigraph by \Cref{obs:rededgeonly}.
We thus conclude by choosing $X_1=L_3$, $X_2=L_2$, $X_3=L_1$ and~$X_4=Z$.

Next, we address the case when exactly one of $\{L_1,L_2\}$ or $\{R_1,R_2\}$  contains $W$ such that $|W| < t$.
Without loss of generality assume that  $W \in \{R_1,R_2\}$ and so both $L_1$ and $L_2$ have size at least $t$. If $W = R_1$, set $\ca Z':= \{L_2,L_1,Z\}$, otherwise set $\ca Z':=\{L_2,L_1,Z,R_1\}$.
Every path $Q\in\ca Q$, by its minimality, connects a vertex of $Z$ to a vertex of $Y\in\ca P^m\setminus\ca Z'$ where $Y$ is a neighbour (red or black) of the set $\ca Z'$ in the partitioned trigraph of $(G,\ca P^m)$.
For any $X$ in $\ca Z$, since $|X|>t$, the union of the parts adjacent to $X$ by black edges in the partitioned trigraph is of cardinality less than~$t$, or we have a $K_{t,t}$ subgraph.
Together, at most $4t$ of the paths of $\ca Q$ in $G$ may end in parts $Y\in\ca P^m\setminus\ca Z'$ which are not red neighbours of $\ca Z'$ in the partitioned trigraph, and at most $t$ paths of $Q$ can go through $W$. Ignoring all these at most $4t$ paths, all the remaining $|Q| - 4t > 4t$ paths have to end in the red neighbor of $L_2$; call this neighbor $L_3$. 
We thus again conclude by choosing $X_1=L_3$, $X_2=L_2$, $X_3=L_1$ and~$X_4=Z$.

Finally, we consider the case when there is $W_L \in \{L_1,L_2\}$ with $|A| < t$ and $W_r \in \{R_1,R_2\}$ with $|W_R|<t$. We will show that this leads to a contradiction.
We first fix the choice of $W_L$ and $W_R$ more precisely. 
If both $L_1$ and $L_2$ have size less than $t$, then we choose $L_1$ as $W_L$ and similarly, if both $R_1$ and $R_2$ have size less than $t$, we choose $R_1$ as $W_R$.
Then in the path $L_2 L_1 Z R_1 R_2$ all parts between $W_L$ and $W_R$ have size at least $t$. 
Since each part $X$ between $W_L$ and $W_R$ has size at least $t$, the union of the parts adjacent to $X$ by black edges has size less than $t$, as otherwise we have a $K_{t,t}$ as a subgraph. 
Thus, the union of $W_L$, $W_R$, and all black neighbors of (at most $3$) parts between $W_L$ an $W_R$ has size at most $5t$. 
This means that there are at most $5t$ vertices which separate $\{a_1,\ldots,a_k\}$ from $\{b_1,\ldots, b_k\}$. Since there are $k > 5t$ disjoint paths between $A:=\{a_1,\ldots,a_k\}$ and $B:=\{b_1,\ldots, b_k\}$, this is a contradiction to Menger's theorem.
\end{proof}

\begin{observation}
\label{obs:pathlayout}
From \Cref{obs:rededgeonly} it follows that
the $s$ paths from $X_1$ to~$X_4$ claimed in \Cref{lem:step1} must each intersect also $X_2$ and $X_3$ (possibly many times there and back).
\end{observation}

\subsection{Proving Step II}

For our convenience, we introduce the following notation.
Given a~graph $G$, an~uncontraction sequence $(G, \ca P^i)$ for $i = 1, \dots, |V(G)|$, an~integer $j \in \{1, \dots, |V(G)|\}$, and $X \in \ca P^j$, we define $N_j^b(X)$ as the set of parts of $\ca P^j$ that have a~black edge to $X$ in the partitioned trigraph of $(G, \ca P^j)$; we call this set \emph{the black neighborhood of $X$ in $(G, \ca P^j)$}.
Also, we set $\| N_j^b(X) \| := \sum_{Y \in N_j^b(X)} |Y|$ to be the number of vertices of $G$ in any part of the black neighborhood of $X$ in $(G, \ca P^j)$.



\begin{lemma}\label{lem:step2}
Let $G$ be an arbitrary simple graph not containing a $K_{t,t}$ as a~subgraph, and let $(G,\ca P^i)$ for $i=1,\ldots,|V(G)|$ be an uncontraction sequence for $G$.
Suppose that, for some $m\in\{1,\ldots,|V(G)|\}$, there are parts $X_1,X_2,X_3,X_4\in\ca P^m$ which induce a red path in this order in the partitioned trigraph of $(G, \ca P^m)$, there is no edge between $X_1$ and $X_4$, and the set $X_1 \cup X_2 \cup X_3 \cup X_4$ induces in $G$ a subgraph containing $s \geq 4t$ pairwise vertex-disjoint paths from $X_1$ to $X_4$.
Then the width of this uncontraction sequence is greater than $2$.
\end{lemma}

\begin{proof}
For a contradiction, suppose that the width of the considered uncontraction sequence of $G$ is at most~$2$.
We are going to formulate an invariant which is true for $(G,\ca P^m)$, and which will remain true at every subsequent step of the uncontraction sequence.
Since this invariant, at the same time, will preclude the finest partition into singletons, the assumed sequence of width~$\leq2$ cannot exist.

\medskip

\textbf{Invariant.}
At step $j\geq m$ of the uncontraction sequence, in the graph $(G,\ca P^j)$ and its partitioned trigraph, the following holds.
There are 4 parts $X_1, X_2, X_3, X_4\in\ca P^j$, each of size at least $t$, forming a red path in this order in the partitioned trigraph of $(G, \ca P^j)$, and parts $X_1$ and $X_4$ are not adjacent.
Denote by $s_j$ the maximum number of vertex-disjoint paths in $G[X_1 \cup X_2 \cup X_3 \cup X_4]$, starting in $X_1$ and ending in $X_4$.
Then $s_j + \| N_{j}^{b}(X_2) \| + \| N_{j}^{b}(X_3) \| \geq 4t$.


\medskip

Note that in the base case we have $s_m \geq 4t$, so the invariant is trivially satisfied for $j = m$.
Also, without loss of generality we can assume that all the vertex-disjoint paths in $G[X_1 \cup X_2 \cup X_3 \cup X_4]$ are inclusion-wise minimal; so in particular, each path contains exactly one (starting) vertex in $X_1$ and exactly one (ending) vertex in $X_4$.
See also an informal illustration in \Cref{fig:X1toX4example}.

Now suppose the invariant holds for some $j \in \{m, \dots, |V(G)| - 1\}$ and prove that it is also preserved after the $j$-th uncontraction.
First observe that for each $i \in \{1, 2, 3, 4\}$, $G$ contains a~bipartite clique with sides $X_i$ and $\bigcup N_j^b(X_i)$. Since $|X_i| \geq t$ and $G$ does not contain $K_{t,t}$ as a~subgraph, we have $\|N_j^b(X_i)\| \leq t - 1$.
Therefore, $s_j \geq 4t - (t - 1) - (t - 1) \geq 2t$.
Each of the $s_j$ disjoint paths in $G[X_1 \cup X_2 \cup X_3 \cup X_4]$ must intersect each of the sets $X_1, X_2, X_3, X_4$ (\Cref{obs:pathlayout}), so actually $|X_i| \geq s_j \geq 2t$ for each $i \in \{1, 2, 3, 4\}$.

\begin{figure}[tbh]
	\centering
	\begin{tikzpicture}[scale=1]
\node [circle, draw,fill,color=white!90!red,minimum size=30 ] at (0,0) {};
\node [circle, draw,fill,color=white!90!red,minimum size=30 ] at (2,0) {};
\node [circle, draw,fill,color=white!90!red,minimum size=30 ] at (4,0) {};
\node [circle, draw,fill,color=white!90!red,minimum size=30 ] at (6,0) {};

\node [circle, draw,label={$X_1$},red,minimum size=30 ] at (0,0) (X1){};
\node [circle, draw,label={$X_2$},red,minimum size=30 ] at (2,0) (X2){};
\node [circle, draw,label={$X_3$},red,minimum size=30 ] at (4,0) (X3){};
\node [circle, draw,label={$X_4$},red,minimum size=30 ] at (6,0) (X4){};

\node [circle, draw,fill,color=white!90!black,minimum size=15 ] at (1,-1.5){};
\node [circle, draw,fill,color=white!90!black,minimum size=15 ] at (2,-1.5){};
\node [circle, draw,fill,color=white!90!black,minimum size=15 ] at (3,-1.5){};
\node [circle, draw,fill,color=white!90!black,minimum size=15 ] at (4,-1.5){};

\node [circle,black,fill,inner sep=1] at (0.9,-1.5) {};
\node [circle,black,fill,inner sep=1] at (1.1,-1.4) {};

\node [circle,black,fill,inner sep=1] at (1.9,-1.5) {};
\node [circle,black,fill,inner sep=1] at (2.0,-1.65) {};
\node [circle,black,fill,inner sep=1] at (2.1,-1.4) {};

\node [circle,black,fill,inner sep=1] at (2.9,-1.4) {};
\node [circle,black,fill,inner sep=1] at (3.1,-1.5) {};

\node [circle,black,fill,inner sep=1] at (3.9,-1.5) {};

\node [circle, draw,minimum size=15 ] at (1,-1.5) (nb1){};
\node [circle, draw,minimum size=15 ] at (2,-1.5) (nb2){};
\node [circle, draw,minimum size=15 ] at (3,-1.5) (nb3){};
\node [circle, draw,minimum size=15 ] at (4,-1.5) (nb4){};
\path [draw,decorate,decoration={brace,mirror,amplitude=7}] (nb1.south) -- (nb3.south) node[midway,below,yshift=-5]{$N^b_j(X_2)$};
\path [draw,decorate,decoration={brace,mirror,amplitude=7}] (nb3.south) -- (nb4.south) node[midway,below,yshift=-5]{$N^b_j(X_3)$};

\path [draw,very thick,red] (X1) -- (X2) -- (X3)  -- (X4);

\path [draw,very thick,black] (X1) -- (nb1) -- (X2) -- (nb3) -- (X3) -- (nb4);
\path [draw,very thick,black] (X2) -- (nb2);

\node [circle,green!60!black,fill,inner sep=1.5] at (0.3,0.2) (s1) {};

\node [circle,green!60!black,fill,inner sep=1.5] at (1.8,0.2) (s2) {};

\node [circle,green!60!black,fill,inner sep=1.5] at (2.2,0.2) (s3) {};

\node [circle,green!60!black,fill,inner sep=1.5] at (3.9,0.3) (s4) {};

\node [circle,green!60!black,fill,inner sep=1.5] at (1.9,-0.1) (s5) {};

\node [circle,green!60!black,fill,inner sep=1.5] at (4.1,-0.2) (s6) {};

\node [circle,green!60!black,fill,inner sep=1.5] at (6.2,-0.2) (s7) {};

\path [draw, thick, green!60!black] (s1) -- (s2) -- (s3) -- (s4) -- (s5) -- (s6) -- (s7);

\end{tikzpicture}
	\caption{Illustration of 4 big parts $X_1,X_2,X_3,X_4$ forming red path (red), an example of one of $s_j$ paths (green), and parts in black neighbourhoods $N^b_{j}(X_2)$, $N^b_{j}(X_3)$ (black).}
	\label{fig:X1toX4example}
\end{figure}
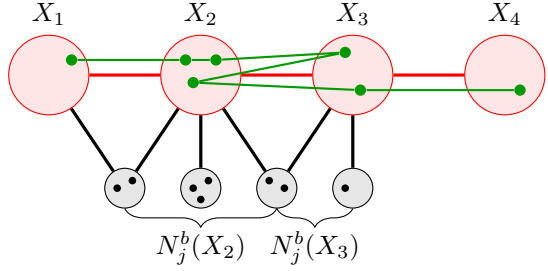



We now analyze all the possible cases for an uncontraction step. Denote by $x\in \ca P^j$ the part that is split 
into new parts $y,z\in \ca P^{j+1}$.
Also select some $s_j \geq 2t$ inclusionwise-minimal pairwise vertex-disjoint paths $P_1, \dots, P_{s_j}$ from $X_1$ to $X_4$ in $G[X_1 \cup X_2 \cup X_3 \cup X_4]$.
By minimality, \Cref{obs:rededgeonly} and \Cref{obs:pathlayout},
each path has its first vertex in $X_1$, its second vertex in $X_2$, its penultimate vertex in $X_3$, and its last vertex in $X_4$.
Also notice that there cannot be any edges between $X_1$ and $X_3$, or any edges between $X_2$ and $X_4$ (otherwise, \Cref{obs:rededgeonly} would apply and the red degree of $X_2$ or $X_3$ would be too large).
Recall also there are no edges between $X_1$ and $X_4$.

If $x\not \in \{X_1,X_2,X_3,X_4\}$, the black neighbourhood of $X_2$ and $X_3$ in the partitioned trigraph can only increase; also, $s_{j+1} = s_j$ as the vertex-disjoint paths in $G[X_1 \cup X_2 \cup X_3 \cup X_4]$ in $(G, \ca P^j)$ are preserved in the uncontraction step.
Hence, the invariant is satisfied.


Suppose $x=X_1$ or $x=X_4$. Without loss of generality, $x=X_1$ (or suppose the red path is actually $X_4, X_3, X_2 ,X_1$), as illustrated in \Cref{fig:X1toX4uncontraction-a}.
Since at the $j$-th step, there were at least $s_j \geq 2t$ inclusion-wise minimal pairwise vertex-disjoint paths from $X_1$ to $X_4$ in $G[X_1 \cup X_2 \cup X_3 \cup X_4]$, at least $t$ of these start in $y$ or at least $t$ of these start in $z$; again without loss of generality, assume that $y$ contains at least $t$ starts of the paths (so naturally $|y| \geq t$).
As each of those paths have its second vertex in $X_2$, we know that $yX_2$ is an~edge in the partitioned trigraph of $(G, \ca P^{j+1})$; and it must be a~red edge, because both $y$ and $X_2$ have size at least $t$ and \Cref{obs:rededgeonly} applies. 
We now claim that the red path $y, X_2, X_3, X_4$ preserves the invariant in $(G, \ca P^{j+1})$. 
Since $y, z \subseteq X_1$, both $y$ and $z$ are non-adjacent to $X_3$ and $X_4$.
Observe also that $z$ cannot be red-adjacent to $X_2$ in $(G, \ca P^{j+1})$ as in this case, the red degree of $X_2$ would be at least $3$.
However, $z$ can be either black-adjacent or non-adjacent to $X_2$.
If $z$ is non-adjacent to $X_2$, then $z$ cannot contain the start of any of the $s_j$ paths $P_i$ as the second vertex of each such path is in $X_2$.
Hence we do not lose any path $P_i$ by removing $z$ from $X_1 \cup X_2 \cup X_3 \cup X_4$, so $s_{j+1} = s_j$ and the invariant still holds.
If $z$ is black-adjacent to $X_2$, then by a~counting argument $z$ contains at most $|z|$ starts of the considered paths $P_i$; hence, $s_{j+1} \geq s_j - |z|$.
However, the number of vertices of $G$ in the black neighborhood of $X_2$ also increases by $|z|$, because $z$ was split out of $X_1$ which was a red neighbour of $X_2$.
Hence the inequality from the invariant is satisfied:
\[ s_{j+1} + \| N_{j+1}^b(X_2) \| + \| N_{j+1}^b(X_3) \| \geq (s_j - |z|) + (\| N_j^b(X_2) \| + |z|) + \| N_j^b(X_3) \| \geq 4t. \]
It is also easy to verify the remaining conditions of the invariant.


It remains to consider $x=X_2$ or $x=X_3$. Without loss of generality, $x=X_2$ (or suppose the red path $X_4,\ldots,X_1$).
We first prove that at least one of $y$, $z$ is red-adjacent to $X_1$ in $(G, \ca P^{j+1})$.
Assume otherwise; then $y$ is either non-adjacent to $X_1$ (and then $|N_G(X_1) \cap y| = 0$) or black-adjacent to $X_1$ (and then $|y| \leq t - 1$ by \Cref{obs:rededgeonly}, so $|N_G(X_1) \cap y| \leq t - 1$).
Analogously, $|N_G(X_1) \cap z| \leq t - 1$.
Since $X_2 = y \cup z$, we conclude that $|N_G(X_1) \cap X_2| \leq 2t - 2$.
But each of $s_j \geq 2t$ vertex-disjoint paths $P_1, \dots, P_{s_j}$ have two consecutive vertices in $X_1$ and $X_2$, so $|N_G(X_1) \cap X_2| \geq s_j$ -- a~contradiction.
Hence at least one of $y$, $z$ is red-adjacent to $X_1$.
Repeating the same argument with $X_3$ instead of $X_1$ implies that at least one of $y$, $z$ is red-adjacent to $X_3$.
Additionally, $X_3$ is already red-adjacent to $X_4$, so exactly one of $y$, $z$ is red-adjacent to $X_3$.
Without loss of generality, assume that $y$ is red-adjacent to $X_3$.
We now consider two separate cases, depending on whether $y$ is red-adjacent to $X_1$.

First assume that $y$ is red-adjacent to $X_1$, as illustrated in \Cref{fig:X1toX4uncontraction-b}.
We claim that the red path $X_1, y, X_3, X_4$ preserves the invariant in $(G, \ca P^{j+1})$.
Observe that $z$ cannot be red-adjacent to $y$ due to the red degree condition of $y$ in $(G, \ca P^{j+1})$.
If $z$ is non-adjacent to $y$ and $X_3$, no inclusion-wise minimal path from $X_1$ to $X_4$ in $G[X_1 \cup X_2 \cup X_3 \cup X_4]$ can be routed through $z$ as all neighbors of $z$ in $X_1 \cup X_2 \cup X_3 \cup X_4$ are in $X_1$. Then $s_{j+1} = s_j$ and it is easy to verify the remaining parts of the invariant.
Next suppose $z$ is black-adjacent to $X_3$; then by \Cref{obs:rededgeonly} we have $|z| \leq t-1$, hence $|y| = |X_2| - |z| \geq t$.
Also observe that at most $|z|$ paths $P_i$ intersect $z$, so $s_{j+1} \geq s_{j} - |z|$.
Therefore, the inequality from the invariant is preserved:
\[ s_{j+1} + \| N_{j+1}^b(X_2) \| + \| N_{j+1}^b(X_3) \| \geq (s_j - |z|) + \| N_j^b(X_2) \| + (\| N_j^b(X_3)  + |z|) \| \geq 4t, \]
and the remaining conditions of the invariant are easy to verify.
Finally suppose $z$ is non-adjacent to $X_3$ and black-adjacent to $y$.
Then each path $P_i$ must intersect $y$, hence $|y| \geq s_j > t$.
Moreover, $s_{j+1} \geq s_j - |z|$ and again $\| N_{j+1}^b(X_2) \| \geq \| N_j^b(X_2) \| + |z|$, so the inequality from the invariant is preserved; once again, the remaining parts of the invariant follow.

It remains to consider the case where $y$ is not red-adjacent to $X_1$ (i.e., either non-adjacent or black-adjacent to $X_1$), as illustrated in \Cref{fig:X1toX4uncontraction-c}.
We claim that $z, y, X_3, X_4$ is a~red path preserving the invariant in $(G, \ca P^{j+1})$.
We first prove that $|y|, |z| \geq t$.
First assume that $|y| \leq t - 1$.
Then $|z| = |X_2| - |y| \geq 2t - (t - 1) \geq t$, so by \Cref{obs:rededgeonly}, $z$ is not black-adjacent to $X_3$; and it is not red-adjacent to $X_3$ as $X_3$ is already red-adjacent to $y$ and $X_4$.
Thus $z$ is non-adjacent to $X_3$, so each path $P_i$ must intersect $y$ (this is because $X_2 = y \cup z$ and each path $P_i$ contains an~edge between $X_2$ and $X_3$).
But then $|y| \geq s_t > t - 1$ -- a~contradiction.
Similarly, if $|z| \leq t - 1$, then $|y| = |X_2| - |z| \geq t$, so by \Cref{obs:rededgeonly}, $y$ is not black-adjacent to $X_1$; and it is not red-adjacent to $X_1$ by our assumption.
So $y$ is non-adjacent to $X_1$ and each path $P_i$ must intersect $z$ -- a~contradiction.
Hence, $|y|, |z| \geq t$.
Applying \Cref{obs:rededgeonly} three times, we find that $y$ is non-adjacent to $X_1$; $z$ is non-adjacent to $X_3$; and $y$ is not black-adjacent to $z$.
Since $y$ must be adjacent to $z$ (otherwise $X_1$ and $X_4$ would be in separate connected components of $G[X_1 \cup X_2 \cup X_3 \cup X_4]$), we get that $y$ is red-adjacent to $z$.
Hence the subgraph of the partitioned trigraph of $(G, \ca P^{j+1})$ induced by $\{X_1, y, z, X_3, X_4\}$ contains red edges $X_1y$, $yz$, $zX_3$ and $X_3X_4$ and no black edges.
Therefore, the second vertex of each path $P_i$ is actually in $y$; so we can remove a~prefix from each path $P_i$ so that each path starts in $y$, finishes in $X_4$ and is contained in $y \cup z \cup X_3 \cup X_4$.
This witnesses that there exist $s_j$ vertex-disjoint paths from $y$ to $X_4$ in $G[y \cup z \cup X_3 \cup X_4]$, so $s_{j+1} \geq s_j$. Moreover, $\|N_{j+1}^b(z)\| = \|N_j^b(X_2)\|$ (as the black neighborhoods of $y$ and $z$ in $(G, \ca P^j)$ are equal to the black neighborhood of $X_2$ in $(G, \ca P^j)$), and $\|N_{j+1}^b(X_3)\| = \|N_j^b(X_3)\|$ (as the black neighborhood of $X_3$ did not change during the uncontraction).
We conclude that $s_{j+1} + \|N_{j+1}^b(z)\| + \|N_{j+1}^b(X_3)\| \geq 4t$, as required, and all the satisfaction of the remaining parts of the invariant is clear.

\begin{figure}[tbh]
	\centering
	\begin{subfigure}[t]{0.49\textwidth}
	\centering
	\resizebox{0.9\textwidth}{!}{
	\begin{tikzpicture}[scale=0.9]
	\node [circle, draw,fill, color=white!90!red, minimum size=25] at (0,0.7) (y) {};
	\node [circle, draw,fill, color=white!90!red, minimum size=25] at (0,-0.7) (z) {};
	\node [circle, draw,fill,color=white!90!red,minimum size=30 ] at (2,0) {};
	\node [circle, draw,fill,color=white!90!red,minimum size=30 ] at (4,0) {};
	\node [circle, draw,fill,color=white!90!red,minimum size=30 ] at (6,0) {};
	
	\node [ellipse, draw,label={$x=X_1$},red,dashed,minimum size=45,fit= (y) (z) ] at (0,0) (X2){};
	\node [circle, draw,minimum size=25 ] at (0,0.7) {$y$};
	\node [circle, draw,minimum size=25 ] at (0,-0.7) {$z$};
	\node [circle, draw,label={$X_2$},red,minimum size=30 ] at (2,0) (X2){};
	\node [circle, draw,label={$X_3$},red,minimum size=30 ] at (4,0) (X3){};
	\node [circle, draw,label={$X_4$},red,minimum size=30 ] at (6,0) (X4){};
	\path [draw,very thick,red] (z) -- (y) -- (X2) -- (X3)  -- (X4);
	\path [draw,very thick,dotted] (z.75) -- (y.-75);
	\path [draw,very thick] (z.100) -- (y.-100);
	\path [draw,very thick] (z) -- (X2);
	\path [draw,very thick,dotted] (z.5) -- (X2.210);
	
	\end{tikzpicture}
	}
	\caption{}
	\label{fig:X1toX4uncontraction-a}
	\end{subfigure}
	\begin{subfigure}[t]{0.49\textwidth}
	\centering
	\resizebox{0.9\textwidth}{!}{
	\begin{tikzpicture}[scale=0.9]
	\node [circle, draw,fill,color=white!90!red,minimum size=30 ] at (0,0) {};
	\node [circle, draw,fill, color=white!90!red, minimum size=25] at (2,0.7) (y) {};
	\node [circle, draw,fill, color=white!90!red, minimum size=25] at (2,-0.7) (z) {};
	\node [circle, draw,fill,color=white!90!red,minimum size=30 ] at (4,0) {};
	\node [circle, draw,fill,color=white!90!red,minimum size=30 ] at (6,0) {};
	
	\node [circle, draw,label={$X_1$},red,minimum size=30 ] at (0,0) (X1){};
	\node [ellipse, draw,label={$x=X_2$},red,dashed,minimum size=45,fit= (y) (z) ] at (2,0) (X2){};
	\node [circle, draw,minimum size=25 ] at (2,0.7) {$y$};
	\node [circle, draw,minimum size=25 ] at (2,-0.7) {$z$};
	\node [circle, draw,label={$X_3$},red,minimum size=30 ] at (4,0) (X3){};
	\node [circle, draw,label={$X_4$},red,minimum size=30 ] at (6,0) (X4){};
	\path [draw,very thick,red] (X1) -- (y) -- (X3)  -- (X4);
	\path [draw,very thick,red] (X1.-10) -- (z.150);
	\path [draw,very thick] (X1) -- (z) -- (X3);
	\path [draw,very thick] (y.-85) -- (z.85);
	\path [draw,very thick,dotted] (y.-95) -- (z.95);
	\path [draw,very thick,dotted] (X1.-30) -- (z.175);
	\path [draw,very thick,dotted] (z.5)-- (X3.210);
	
	\end{tikzpicture}
	}
	\caption{}
	\label{fig:X1toX4uncontraction-b}
	\end{subfigure}
	\begin{subfigure}[t]{0.45\textwidth}
	\centering
	\resizebox{\textwidth}{!}{
	\begin{tikzpicture}[scale=0.9]
	\node [circle, draw,fill,color=white!90!red,minimum size=30 ] at (0,0) {};
	\node [circle, draw,fill, color=white!90!red, minimum size=25] at (2,0.7) (y) {};
	\node [circle, draw,fill, color=white!90!red, minimum size=25] at (2,-0.7) (z) {};
	\node [circle, draw,fill,color=white!90!red,minimum size=30 ] at (4,0) {};
	\node [circle, draw,fill,color=white!90!red,minimum size=30 ] at (6,0) {};
	
	\node [circle, draw,label={$X_1$},red,minimum size=30 ] at (0,0) (X1){};
	\node [ellipse, draw,label={$x=X_2$},red,dashed,minimum size=45,fit= (y) (z) ] at (2,0) (X2){};
	\node [circle, draw,minimum size=25 ] at (2,0.7) {$y$};
	\node [circle, draw,minimum size=25 ] at (2,-0.7) {$z$};
	\node [circle, draw,label={$X_3$},red,minimum size=30 ] at (4,0) (X3){};
	\node [circle, draw,label={$X_4$},red,minimum size=30 ] at (6,0) (X4){};
	\path [draw,very thick,red] (X1) -- (z) -- (y) -- (X3)  -- (X4);
	\path [draw,very thick,dotted] (X1) -- (y);
	\path [draw,very thick,dotted] (X3) -- (z);
	
	\end{tikzpicture}
	}
	\caption{}
	\label{fig:X1toX4uncontraction-c}
	\end{subfigure}
	 \caption{Illustration of possible uncontractions from the proof of \Cref{lem:step2}. Multiple edge connections between parts show all possible edges, dotted edge means non-adjacency.}
		\label{fig:X1toX4uncontractions}
	\end{figure}
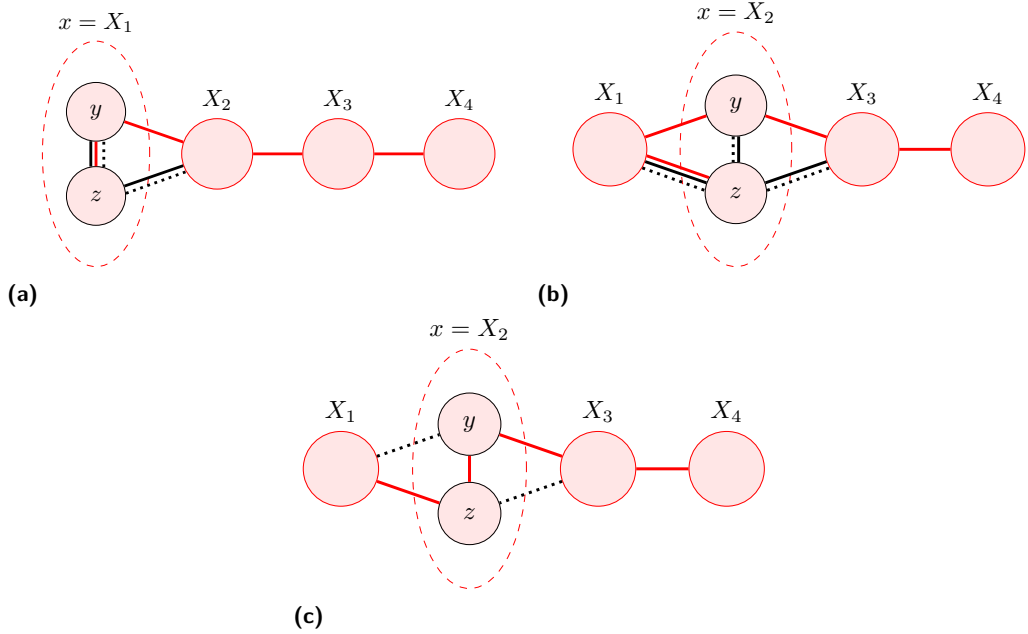

\end{proof}

\subsection{Concluding the Main Proof}

Observe that the basic assumptions of \Cref{thm:mainalt} are the same as those of \Cref{lem:step1}, and the assumptions of \Cref{lem:step2} follow from \Cref{lem:step1}.
Hence, as detailed in the formal proof below, \Cref{lem:step1} and \Cref{lem:step2} together imply \Cref{thm:mainalt}.

To extend this to a proof of \Cref{thm:main}, we use the following tool -- an excluded-grid theorem for tree-width in the currently strongest published formulation (modulo polylog factors which we have ``rounded up'' for simplicity):

\begin{theorem}[Chuzhoy and Tan~{\cite{DBLP:journals/jctb/ChuzhoyT21}}]\label{thm:grid10}
There is a function $f(n)\in \ca O(n^{10})$ such that, for every positive integer $N$, 
if a graph $G$ is of tree-width at least $f(N)$, then $G$ contains a subdivision of the $N\times N$ wall as a subgraph.
\end{theorem}

\begin{proof}[Proof of \Cref{thm:main}]
Let $f$ be the function of \Cref{thm:grid10}, and choose $f_1(t)=f(2\cdot16\cdot(13t)^2+2)\in\ca O(t^{20})$.
If our graph $G$ has tree-width at least $f_1(t)$, then $G$ contains a subdivision of the $(2\cdot16\cdot(13t)^2+2)\times(2\cdot16\cdot(13t)^2+2)$ wall by \Cref{thm:grid10}.
Consequently, by \Cref{obs:wall2mesh}, $G$ contains a subgraph which is a $(16\cdot(13t)^2)\times(16\cdot(13t)^2)$ cubic mesh.

We now invoke \Cref{lem:step1}, assuming $G$ admits an uncontraction sequence of width at most~$2$;
so, we obtain the parts $X_1,X_2,X_3,X_4$ as claimed by \Cref{lem:step1} and required by \Cref{lem:step2}.
By the subsequent invocation of \Cref{lem:step2}, we immediately arrive at a contradiction to having the uncontraction sequence of width at most~$2$.
Consequently, the tree-width less than $f_1(t)=f(2\cdot16\cdot(13t)^2+2)\in\ca O(t^{20})$.
\end{proof}

\begin{proof}[Proof of \Cref{cor:tww2seq}]
Let $t$ be an integer and $\ca C$ a graph class such that no $G \in \ca C$ contains $K_{t,t}$ as a subgraph. 
Let $f_1$ be the polynomial function from (the proof of) \Cref{thm:main} and set $k:=f_1(t)$, which is a constant depending on $\ca C$. 
Let $G \in \ca C$ be the input graph. 
We first use the linear time algorithm of Bodlaender~\cite{DBLP:journals/siamcomp/Bodlaender96} to test whether $G$ has tree-width at most $k$. If the answer is no, then we know by \Cref{thm:main} that $G$ cannot have twin-width at most $2$. 
Otherwise, $G$ has tree-width at most $k$, and we easily turn the decomposition into a branch-decomposition of width at most~$k+1$.
This directly gives a boolean-width decomposition of width at most $k+1$~\cite{DBLP:conf/wg/AdlerBRRTV10} (with virtually the same decomposition tree). 
Finally, by the result of~\cite{DBLP:journals/jacm/BonnetKTW22}, from a boolean-width decomposition of $G$ of width $k+1$ one can obtain a contraction sequence of width at most $2^{k+2}-1$ and an easy inspection of the proof shows that this can be done in polynomial time. 

Thus, in the end (with respect to the fixed class $\ca C$) we compute in polynomial time a contraction sequence of $G$ of width at most $2^{\text{poly$(t)$}}$.
If the class $\ca C$ and implicit $t$ were to be considered as parameters, the discussed algorithm would have an FPT runtime.
\end{proof}


\subsection{Case of Twin-width 3}

Lastly, we show that in the class of graphs of twin-width~$3$, no upper bound on the tree-width is possible even if we exclude~$K_{2,2}$.


\begin{lemma}
	For every positive integer $N $, there exists an $(N^2+N)$-vertex graph with no $K_{2,2}$ subgraphs whose twin-width is at most 3  and tree-width is at least $N$.
\end{lemma}

\begin{figure}[tbh]
	\centering
	\includegraphics[width=0.7\linewidth]{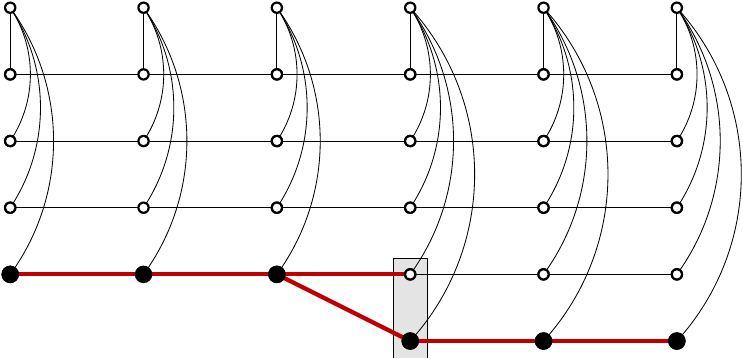}
	\caption{Illustration of the $K_{2,2}$-free graph of twin-width 3 and tree-width $\theta(\sqrt{n})$.
	It represents a partitioned trigraph of a contraction sequence of width 3: white-filled vertices are singletons and the gray rectangle reveals the next contraction.}
	\label{fig:tww3example}
\end{figure}

\begin{proof}
	Let $N$ be a positive integer.
	We take $G$ the graph obtained from the disjoint union of $N$ paths $P_1,\dots,P_N$ of length $N$, and for each $i\in [N]$, we add a new vertex adjacent to the $i$-th vertex of each path $P_1,\dots,P_N$.
	See \Cref{fig:tww3example} for an illustration.
	Obviously $G$ has $N^2+N$ vertices and no $K_{2,2}$ as subgraph.
	
	Observe that by contracting each path $P_1,\dots,P_N$, we obtain the complete bipartite graph $K_{N,N}$.
	As $G$ admits $K_{N,N}$ as a minor, we deduce that its tree-width is at least $N$.
	
	We claim that the twin-width of $G$ is at most 3. We can assume without loss of generality that the edges of $P_1$ are red.
	We contract $P_1$ and $P_{2}$ by iteratively contracting their $i$-th vertices as illustrated in \Cref{fig:tww3example}.
	The maximum red degree of each encountered trigraph is $3$ and we end up with a trigraph isomorphic to $G-P_{2}$.
	We can repeat this operation to end up with a trigraph isomorphic to $G- (P_2\cup \dots \cup P_N) $.
	Then we can contract each pendant vertex with its unique neighbor to obtain a red path whose twin-width is obviously at most~$2$.
	We conclude that the twin-width of $G$ is at most $3$.
\end{proof}

\section{Conclusions}


We have shown that sparse classes of graphs of twin-width $2$ have bounded tree-width. This means that the existing rich machinery of structural and algorithmic tools for tree-width is applicable to such graph classes. 

One might wonder how one can relax the requirement on $\ca C$ being $K_{t,t}$-free to relate graphs of twin-width $2$ to other well-studied graph notions.
One could even try to drop any restrictions altogether and conjecture that any class of twin-width at most $2$ has bounded clique-width. This cannot be true, since the class of unit interval graphs has twin-width $2$ and unbounded clique-width \cite{gr00}. However, we conjecture the following.


\begin{conjecture}
Let $\ca C$ be a stable class of graphs of twin-width at most $2$. Then $\ca C$ has bounded clique-width.
\end{conjecture}

Here being stable means that there exists $k$ such that no $G \in \ca C$  contains a half-graph  of order $k$ as a semi-induced subgraph. We refer to~\cite{stable_rw,stable_tww} for details about stability and a discussion on how it relates to tree-width, clique-width and twin-width.

\bibliography{tww}

\end{document}